\documentclass[12pt,a4paper]{amsart} 
\usepackage{graphicx,amssymb}
\usepackage{mathrsfs}
\usepackage{amssymb,amsmath,amsthm,color}
\usepackage{graphicx}
\usepackage{hyperref}
\usepackage{url} 
\usepackage{setspace}

\newtheorem{thm}{Theorem}[section]
\newtheorem{cor}[thm]{Corollary}

\numberwithin{equation}{section}

\newcommand{\R}{\mathbb R}
\newcommand{\C}{\mathbb C}
\newcommand{\Z}{\mathbb Z}
\newcommand{\T}{\mathbb T}

\def\TagOnRight

\def\R {\mathbb{R}}

\newcommand{\be}{\begin{equation}}
\newcommand{\ee}{\end{equation}}
\newcommand{\bea}{\begin{eqnarray}}
\newcommand{\eea}{\end{eqnarray}}
\newcommand{\Bea}{\begin{eqnarray*}}
\newcommand{\Eea}{\end{eqnarray*}}
\newcommand{\bt}{\begin{Theorem}}
\newcommand{\et}{\end{Theorem}}

\newcommand{\bpr}{\begin{Proposition}}
\newcommand{\epr}{\end{Proposition}}

\newcommand{\bl}{\begin{Lemma}}
\newcommand{\el}{\end{Lemma}}
\newcommand{\bi}{\begin{itemize}}
\newcommand{\ei}{\end{itemize}}

\newtheorem{Definition}{Definition}[section]
\newtheorem{Theorem}[Definition]{Theorem}
\newtheorem{Lemma}[Definition]{Lemma}
\newtheorem{Proposition}[Definition]{Proposition}

\newtheorem{Remark}[Definition]{Remark}

\usepackage{mathrsfs}
\usepackage{amsthm}
\usepackage{hyperref}
\usepackage{comment}

\begin{document}
\baselineskip16pt

\title[Functions operating on $M^{p,1}$ and nonlinear dispersive equations]{Functions operating on modulation spaces and nonlinear dispersive equations}
\author{Divyang G. Bhimani, P. K. Ratnakumar  }
\address {Harish-Chandra Research Institute, Allahabad-211019,
 India}
 \email { divyang@hri.res.in, ratnapk@hri.res.in}

\maketitle
\begin{abstract}
The aim of this paper is two fold. We show that if a complex function $F$ on $\C$
operates in  the modulation spaces $M^{p,1}(\R^n)$ by composition, 
then $F$ is real analytic on $\R^2 \approx \C$. This answers 
negatively, the open question posed in [M. Ruzhansky, M. Sugimoto, B. Wang,  Modulation Spaces and Nonlinear Evolution Equations, arXiv:1203.4651], regarding the 
general power type nonlinearity of the form $|u|^\alpha u$. We also characterise the 
functions that operate in the modulation space $M^{1,1}(\R^n)$. 

The local well-posedness of the NLS, NLW and NLKG equations for 
the `real entire' nonlinearities are also studied in some weighted modulation spaces $M^{p,q}_s(\R^n)$.

\end{abstract}
\section{Introduction}

A classical theorem of Wiener \cite{wiener} and L\'evy \cite{levy} tells us that if $f\in A(\mathbb T)$, where $A(\mathbb T) $ denotes the class of all functions on the unit circle whose Fourier series is absolutely convergent, and $F$ is defined and analytic on the range of $f,$ then $F(f) \in A(\mathbb T);$ where $F(f)$ is  the composition of functions $F$ and $f.$ 
Recently in this direction, M. Sugimoto et al. has shown in \cite{nonom}, that the Wiener-L\'evy type theorem is valid  in the realm of weighted modulation spaces $M^{p,q}_s (\R^n), 1\leq p,q \leq \infty, s> n/q', 1/q+1/q'=1$, for complex  entire function $F$ vanishing at $0$.  See section 2, for the definition of modulation spaces.

For $F: \R^2 \to \C$, we denote by $T_F$ the operator $T_F(f):=F(f)$.  One of the results that, we prove in this paper, namely Theorem \ref{convs}, says that if $F$ is real entire and $F(0)=0$, then $T_F$ maps  $X$ to itself, that is, $T_{F}(X)\subset X$; where 
$X=M^{p,1}(\mathbb R^{n}), (1\leq p \leq \infty) $ or $X=M^{p,q}_{s}(\mathbb R^{n}), (1\leq p,q \leq \infty, s>n/q');$ here the proof  relies on the multiplication algebra property of $X.$
In fact, we have gone a bit further, and shown that under the weaker hypothesis on $F,$ namely that $F$ is real analytic on $\mathbb R^{2}$ and $F(0)=0,$ then $T_{F}\left(M^{1,1}(\mathbb R^{n})\right)\subset M^{1,1}(\mathbb R^{n});$ the proof relies on the  invariant property of the modulation space $M^{1,1}(\mathbb R^{n})$ under the Fourier transform. This invariance is not available for $M^{p,1}(\R^n)$, when  $p > 1$.

Now the natural question is: What are all functions $F$ such that $T_F$ takes $M^{p,1}(\mathbb R^{n})$ to itself ?   In Theorem \ref{orao}, we show that  $T_F$ maps $M^{p,1}(\R^n) $ to itself, then $F$ must be real analytic on $\R^2$. The proof relies on the ``localized" version of ``time-frequency" spaces, which can be identified with the Fourier algebra on the torus $A(\T^n)$.  As a consequence of the above two results, we 
obtain a characterisation theorem for the functions $F$ that operates in $M^{1,1}(\R^n).$  In fact,  we show that  the real analytic functions are the only ones with the desired mapping property on $M^{1,1}(\R^n)$,  see Theorem \ref{re}.  

Our motivation for studying the above problem started with analysing well-posedness results for nonlinear  Schr\"{o}dinger equation $(NLS)$ for power type nonlinearities in Lebsgue spaces $L^{p}(\mathbb R^{n})$. 
Consider the initial value problem 
\begin{equation*}
 (NLS) \ \ \ i\frac{\partial }{\partial t}u(x,t)+\bigtriangleup_{x}u(x,t)= F(u(x,t)), ~~ u(x, 0)=u_{0}(x),
\end{equation*}
where  $ \bigtriangleup_{x}=  \sum_{j=1}^{n} \frac{\partial^{2}}{\partial x_j^2 }, (x,t) \in \R^n\times \mathbb R, i=\sqrt{-1}, $   $u_{0}$ is a complex valued function on $\R^n$ and the nonlinearity is given by a complex function $F$ on $\C$. 

The theory of $NLS$ and  other general nonlinear dispersive equations (\ref{nlw},
\ref{nlkg}) (local and global existence) is vast and has been studied extensively by many authors, see for instance \cite{gv}.
Almost exclusively, well-posedness has been established in energy space $H^1(\R^n)$ and $H^s(\R^n)$.  The techniques have been restricted to $L^2$ Sobolev spaces  because of the crucial role played by the Fourier transform  in the analysis of partial differential equations. 

The modern approach to study the well-posedness  of dispersive equations, is through the Strichartz estimates  \cite{rss}, satisfied by the linear propagator for the associated dispersive equation. We refer to the far reaching generalisation of Strichartz estimates in \cite {kt} by M. Keel and T. Tao,  and also \cite{tao, tc} and the references therein.

However, much less is known for the same problem with $L^p$ initial data, for $p \neq 2.$ The reason is the non availability of analogous Strichartz estimate for $L^p$ functions, for $p \neq 2$. Perhaps the first attempt 
in this direction started  in the work of Vargas and Vega \cite{vv}, where they studied the Schr\"{o}dinger equation 
with initial data having infinite $L^2$ norm. There are also other  attempts by considering the class of functions  like ${\mathcal F}(L^p)$ instead of $L^p$, see \cite{hr}.

An interesting subclass of $L^p(\mathbb R^{n})$ is the modulation space $M^{p,1}(\mathbb R^{n})$, for $1 \leq p \leq \infty$.
In fact, the modulation spaces have been attracted by many mathematicians,  working in the field of partial differential equations. The modulation spaces  provide an excellent  substitute and having interesting properties, that are known to 
fail on Lebesgue spaces.  For instance, the multiplier   $e^{it4 \pi^2|\xi|^{2}}$ corresponding to the  Schr\"odinger propagator $e^{-it \Delta}$, provides a bounded Fourier multiplier operator on $L^{p} (\R^n)$ only for $p=2$. This is a classical theorem 
of H\"ormander \cite{Hor, Leb}, where as all the multipliers  $e^{it|\xi|^{a}},~ 0\leq a \leq  2$ defines bounded Fourier multiplier operators on  $M_{s}^{p,q}(\R^n)$ for $1\leq p,  q \leq \infty, s\geq 0 $, as shown by B\'enyi-Gr\"ocheing-Okoudjou-Rogers in \cite{benyi}, see also \cite{ambenyi}. \\

Now we recall some well-posedness results on modulation spaces. The local well-posdeness of  $NLS$ in $M^{2,1} (\R^n)$, is a result of Baoxiang-Lifeng-Boling \cite{bao}, Theorem 1.1 and 1.2, with nonlinearities that includes the power-like; $F_{k}(u)=\lambda |u|^{2k} u, (k\in \mathbb N, \lambda \in\mathbb R)$ as well as the exponential-like; $F_{\rho}(u)=\lambda (e^{\rho |u|^{2}}-1)u, (\lambda, \rho \in \mathbb R)$. This was generalized by B\'enyi-Okoudjou in  \cite{ambenyi}; and in fact, their results give,    the local well-posdeness of  $NLS(\ref{nls})$,  $NLW$(\ref{nlw}) and $NLKG  (\ref{nlkg})$ in $M_{s}^{p,1}(\R^n)$ for $1\leq p \leq \infty, s\geq 0.$ The  nonlinearities included in their  work have the generic form $F(u)=g(|u|^{2})\, u$,  for some complex-entire function $g(z)$.\\

One of the key points in the above results is that, the above nonlinearities map the modulation space to itself. In fact, the proof of the above local well-posedness results  crucially relies on the fact that  $M^{p,1}_s(\mathbb R^{n})$ is a function algebra under pointwise multiplication: $\| f g \|_{M^{p,1}_s} \leq C \|f \|_{M^{p,1}_s} \|g\|_{M^{p,1}_s}$ for some constant $C$. Therefore, if  $\alpha =2k$, $|u|^{\alpha} u = u^{k+1} \bar{u}^k$ and hence 
$ \||u|^{\alpha} u\|_{M^{p,1}_s} \leq C \|u\|_{M^{p,1}_s}^{2k+1}$. Hence the nonlinearity of the type $F(z)=z|z|^{\alpha}, ~\alpha \in 2\mathbb N$ can be handled in this way. Of course it is very natural to ask, how far can one  go, to include more general nonlinear terms in these dispersive equations on modulation spaces? It was in this context M. Ruzhansky,  M. Sugimoto and B. Wang  posed  the open problem in \cite{ruzha}, namely the validity of the inequality $\|f|f|^{\alpha}\|_{M^{p,1}}\leq C\|f\|_{M^{p,1}}^{\alpha +1}$ for all $f\in M^{p,1}(\mathbb R^{n})$ and $\alpha \in (0, \infty)\setminus 2\mathbb N.$ In the present paper, we answer it negatively, see Corollary \ref{op} to Theorem \ref{orao}. In fact, Theorem \ref{orao} answers a much more general question. \\

Our  Theorem \ref{convs} has inspired us to consider nonlinearites of the form,
\begin{eqnarray} \label{nlf}
F(u)= G(u_{1}, u_{2});
\end{eqnarray}
where $u=u_{1}+iu_{2}$ and $G:\mathbb R^{2}\to \mathbb C$ is real entire on $\mathbb R^{2}$ with $G(0)=0.$ This generalizes the nonlinearities previously studied in modulation spaces. With the help of estimate (\ref{favest}) for 
real entire nonlinearities given by Theorem \ref{convs}, and well established Fourier multiplier estimates,  we prove  the  local well-posdeness results of  $NLS$ (\ref{nls}),  $NLW$ (\ref{nlw}), and $NLKG$ (\ref{nlkg}) with Cauchy data in $X,$ where $X$ denotes the spaces  $M_{s}^{p,1}(\mathbb R^{n}), (1\leq p \leq \infty, s\geq 0);$  or $M_{s}^{p,q}(\mathbb R^{n}), (1\leq p, q \leq \infty, s> n/q'),$ see  Theorems \ref{nlsrt}, \ref{pnlw} and Theorem \ref{pkg}.
 \begin{Definition}
We say, a complex function $F$ on $ \mathbb R^2$
 operates in the modulation space  $M^{p,q}_s(\R^n)$, if $F(f_1,f_2)\in M^{p,q}_s(\mathbb R^{n})$ whenever $f= f_1+ if_2\in M^{p,q}_s(\mathbb R^{n}) $, with $f_1 = \text{Re}(f)$.
\end{Definition}
\begin{Remark}
Note that the above is a weak definition in the sense that, it doesn't demand any norm inequality for $F(f)$. 
\end{Remark}

We mainly prove three results in this paper. Theorem \ref{orao} shows that if $F$ operates in the modulation space
$M^{p,1} (\R^n), (1\leq p \leq  \infty)$ then $F$ has to be real analytic on $\R^2$. For a converse, we prove two facts:
If $F$ is real analytic on $\R^2$ and $F(0)=0$, then $F$ operates in $M^{1,1}(\R^n)$. Moreover,  if $F$ is real entire and $F(0)=0$ then $F$ 
operates in $M^{p,1}_s(\R^n) $ for $1\leq p \leq\infty, ~ s\geq 0$ and also in  $ M^{p,q}_{s}(\mathbb R^{n}), $ for $1\leq p,q \leq  \infty, ~s>n/q'$, see Theorem \ref{convs}.  As a corollary of the above two results,
we also obtain the following interesting characterising result for functions operating in $M^{1,1}(\R^n)$:

\begin{Theorem} \label{re}

Let $F$ be a complex valued function on $ \R^2$. 
Then $F$ operates in $M^{1,1}(\R^n)$ if and only if $F$ is real analytic on $\R^2$ and $F(0)=0$. 
\end{Theorem}

We would like to point out that Theorem \ref{orao} throws light on the limitation of the prevailing method of studying well-posedness in modulation spaces $M^{p,1}_s(\R^n)$ using the algebraic property available in these spaces. Our result (Theorem \ref{orao}) shows that this approach using the algebraic property or even the general 
mapping property of the nonlinearity  of the modulation space to itself,  can  
handle only the so-called real analytic nonlinearities on $M^{p,1}(\R^n)$. In particular, the nonlinearities of interest in applications, 
namely the power type $F(u)= |u|^\alpha u$ for  $\alpha \notin 2{\mathbb N}$, and also the exponential type $F(u) = e^{u |u|} -1$ are ruled out in this approach. This leads to the fact that to deal with local  existence for nonlinear Schr\"{o}dinger equation and other dispersive equations  with power type nonlinearity $|u|^\alpha u$ when $\alpha$ is not an even integer, requires some new 
approach.

\section{Modulation spaces}
\label{prlm}
In this section, we briefly discuss the modulation spaces and some relevant properties of these spaces.
Modulations  spaces were introduced during the early eighties in the pioneering work of H.G. Feichtinger \cite{fei1}. 
Subsequently, in a joint work with K. Gr\"ochenig, the basic theory of these function 
spaces were established in  \cite{fei, fei-gr}. Most of the results we discuss in this section can be seen in the book by Gr\"ochenig \cite{gro}. 

Let $\mathcal{S}(\mathbb R^{n})$ and  $\mathcal{S'}(\mathbb R^{n})$ denote the Schwartz space and the space of tempered distributions, respectively. We define the Fourier transform of $f\in \mathcal{S}(\mathbb R^{n})$ by 
\begin{eqnarray}
\mathcal{F}f(w)=\widehat{f}(w)=\int_{\mathbb R^{n}} f(x) \, e^{-2\pi i w\cdot x} dx,~ w\in \mathbb R^{n},
\end{eqnarray}
and the inverse Fourier transform by
\begin{eqnarray}
\mathcal{F}^{-1}f(x)=f^{\vee}(x)=\int_{\mathbb R^{n}} f(w)\, e^{2\pi i x\cdot w} dw,~~x\in \mathbb R^{n}.
\end{eqnarray} 
The Fourier transform is an isomorphism of the Schwartz space $\mathcal{S}(\mathbb R^{n})$ onto itself, and extends to the tempered distributions by duality.
Hence for every tempered distribution $f$, we have $$(\widehat{f} \,\, )^{\vee}=f=\widehat{(f^{\vee})}.$$ 

The modulation spaces are defined in terms of the short time Fourier transform. The short time Fourier transform of a function $f$ with respect to a window function $g \in {\mathcal S}(\R^n)$ is defined by
\bea  V_{g}f(x,w) = \int_{\R^n} f(t) \overline{g(t-x)} \, e^{-2\pi i w\cdot t} \, dt,  ~  (x, w) \in \mathbb R^{2n} \eea
whenever the integral exists. 

For $x, w \in \R^n$ the translation operator $T_x$ and the modulation operator $M_w$ are
defined by $T_{x}f(t)= f(t-x)$ and $M_{w}f(t)= e^{2\pi i w\cdot t} f(t).$ In terms of these
operators 
the short time Fourier transform may be expressed as
\bea \label{ipform} V_{g}f(x,w)=\langle f, M_{w}T_{x}g\rangle,\eea
 where $\langle f, g\rangle$ denotes the inner product for $L^2$ functions,
or the action of the tempered distribution $f$ on the Schwartz class function $g$.  Thus $V: (f,g) \to V_g(f)$ extends to a bilinear form on $\mathcal{S}'(\mathbb R^{n}) \times \mathcal{S}(\mathbb R^{n})$ and $V_g(f)$ defines a uniformly continuous function on $\R^{n} \times \R^n$ whenever $f \in \mathcal{S}'(\R^n) $ and $g \in  \mathcal{S}(\R^n)$.

Since $\overline{M_wT_x g} = e^{-2\pi i x \cdot w} T_x M_w g^*$ with $g^*(y) = \overline{g(-y)}$,
from (\ref{ipform})
we see that the short time Fourier transform can also be expressed as a convolution: \bea \label{conform} V_{g}f(x,w) =e^{-2\pi i x \cdot w} \left( f\ast M_w g^*\right)(x). \eea

\begin{Definition} Let $1 \leq p,q \leq \infty, ~s\geq 0$ and $0\neq g \in{\mathcal S}(\R^n)$. The  Modulation space   $M^{p,q}_s (\mathbb R^{n})$
is defined to be the space of all tempered distributions $f$ for which the following  norm is finite:
\bea \label{norm} \|f\|_{M^{p,q}_{s}}=  \left(\int_{\R^n}\left(\int_{\R^n} |V_{g}f(x,w)|^{p} dx\right)^{q/p} \langle w \rangle_s^{q} \, dw\right)^{1/q},\eea
where $\langle w \rangle_s=(1+|w|^{2})^{s/2}$, for $ 1 \leq p,q <\infty$. If $p$ or $q$ is infinite, $\|f\|_{M^{p,q}_s}$ is defined by replacing the corresponding integral by the essential supremum. When $s=0$, we write $M^{p,q}(\R^n)$ instead of $M^{p,q}_0(\R^n)$.
\end{Definition}

Note that the modulation space norm is the weighted  mixed $L^p$ norm of the short time Fourier transform $ V_{g}f $ on $\mathbb R^{n} \times \mathbb R^{n}$ in $L^{p,q}_s$ with measure $dx \langle w  \rangle_s d w .$  In particular, if $p=q$, the modulation space norm is just the weighted $L^p$ norm of $V_gf$ with weight $\langle w  \rangle_s$. 
The Modulation spaces can also be defined for  exponents $0<p,q <1$, see Kobayashi \cite{kob}, however, we will restrict to the case $1\leq p, q \leq \infty, s\geq 0$.

\begin{Remark}
\label{equidm}
The definition of the modulation space given above, is independent of the choice of 
the particular window function. In fact if $g$ and $g'$ are any two window functions, then we have the relation
$$\| V_{g'}f \|_{L^{p,q}_s} \lesssim \| V_{g'}g \|_{L^{1,1}_s} \| V_{g}f \|_{L^{p,q}_s},$$  see \cite[p.233]{gro}.
It follows that, the modulation space norms given by $g$ and $g^\prime$ are equivalent. Here $A\lesssim B$ stands for $A \leq C B$ for some constant $C$. 
\end{Remark}

Now we recall some relevant properties of the modulations spaces. First of all $M^{p,q}_{s} (\R^n),~1\leq p, q \leq \infty,$ are Banach spaces with respect to the norm given by (\ref{norm})  and contains ${\mathcal S}(\R^n)$ as a dense subspace,  see Theorem 11.3.5 in  \cite{gro}. 

$M^{p,q}_{s}(\R^n)$ is invariant under the so called ``time - frequency" shifts 
$$T_{x_0} M_{w_{0}} : f(x) \mapsto e^{2 \pi i w_0 \cdot (x-x_{0})} f(x-x_0),$$ for $x_0, w_0 \in \R^n$. This invariance  follows immediately from the observation
$$V_{g}(T_{x_0} M_{w_{0} }f) (x, w)= e^{-2\pi iw \cdot x_0}V_{g}f (x-x_0, w-w_0),$$
and the fact that the mixed $L^p$ space,  $L^{p,q}(\R^n \times \R^n)$ are invariant under modulation and translation operators. $M^{p,q}(\mathbb R^{n})$ is invariant under the Fourier transform, when $p=q, 1\leq p < \infty$.
In fact, a straight forward computation gives the identity
$$ V_gf(x, w) = e^{-2 \pi i x \cdot w } \, V_{\widehat{g}} \widehat{f}(w, -x).$$
Choosing $g(x)=e^{-\pi |x|^2}$ so that $g=\widehat{g}$, and taking the $L^p$ norm, we get
\bea \label{ftinv}
\|f \|_{M^{p,p}}= \|\widehat{f} \|_{M^{p,p}}.\eea Thus the Fourier transform defines an isometry on 
$M^{p,p}(\R^n)$, for $1 \leq p < \infty$. Note that the above arguments are valid only for $s=0$.

The modulation spaces are also invariant under complex conjugation:  Since $V_g\bar{f}(x,w) = \overline{V_gf(x,-w)}$ for  real valued $g$, we have $\|\bar{f} \|_{M^{p,q}}= \|f\|_{M^{p,q}}$. From this, it also follows the useful inequality 
\bea \label{reimineq} \|\text{Re}f\|_{M^{p,q}_s} \leq \| f\|_{M^{p,q}_s}, ~ ~ \|\text{Im}f\|_{M^{p,q}_s} \leq \| f\|_{M^{p,q}_s} .\eea

We also have the following embedding results (see \cite{gro}, Theorem 12.2.2 ),
\bea M^{p_{1},q_{1}}_{s} (\R^n)\hookrightarrow M_{s}^{p_{2}, q_{2}}(\R^n),\eea
 if $p_{1}\leq p_{2}$ and $q_{1} \leq q_{2}$.  Also, if $q_{1}\leq q_{2}$, an application of H\"olders inequality,
 shows that 
\bea M^{p,q_{1}}_{s_{1}}(\R^n) \hookrightarrow M^{p, q_{2}}_{s_{2}}(\R^n),\eea
whenever $s_{1}-s_{2} > n/q_{2}-n/q_{1}.$ As a consequence of these two inclusions, we see that  for $s>n/q'$, 
\bea M^{p,q}_{s} (\R^n)\hookrightarrow M^{\infty, 1} (\R^n).\eea

There are several embedding results between Lebesgue, Sobolev, or Besov spaces and modulation spaces, see for example, \cite{embd, sugi}. We note, in particular that the $L^2$ Sobolev space $H^s(\mathbb R^{n})$ coincides with $M^{2,2}_{s}(\mathbb R^{n}).$  We also refer to \cite{baob} for some recent developments in PDEs on Modulation spaces and the references therein. 

Now we prove the following result. The case $p=q=1$, has been observed in \cite{gro}.
\begin{Proposition} \label{convoreslt}
If $k \in L^1(\mathbb R^{n})$ and $f \in M^{p,q}(\mathbb R^{n})$ then $k \ast f \in M^{p,q}(\mathbb R^{n}), ~ 1\leq p,q < \infty.$ Moreover, we have the inequality
\bea
\| k \ast f \|_{M^{p,q}} \leq \| k \|_{L^1} \, \|  f \|_{M^{p,q}}.\eea 
\end{Proposition}

\begin{proof}
The proof follows from the corresponding property for the Lebesgue space $L^1(\mathbb R^{n})$. In fact, by the identity (\ref {conform}) we see that
 \Bea |V_g[k \ast f](x,w)| = |k \ast  f \ast M_w g^*(x)|=  |k \ast [ f \ast M_w g^*](x)|.\Eea Now taking $L^p$ norm with respect to the $x$ variable
and applying the convolution result in $L^p(\R^n, dx)$, and then $L^q$ norm with respect to the variable $w$, gives the result.
 \end {proof}

Next we prove an approximation result on the modulation space $M^{p,1}(\R^n)$ for $1\leq p <\infty.$
Let $\phi \in \mathcal{S}(\mathbb R^{n}),$  with $\int_{\R^n} \phi =1$ and  and set 
$\phi_r(x):= r^{-n} \phi(x/r),  r>0.$ Then the family $\{ \varphi_r \}_{r>0}$ is called an approximate identity in  $M^{p,1}(\mathbb R^{n})$ in view of the next lemma. We use only the case $p=q=1$.

\begin{Lemma} \label{conv} \label{ai}
Let $\{\phi_r\}_{r>0}$ be as above and $f\in M^{p,q}(\mathbb R^{n}), 1\leq p,q <\infty.$ Then  given $\epsilon >0$, there exists a $\delta>0$ such that $\|f\ast \phi_{r}-f\|_{M^{p,q}} <\epsilon $ whenever $r<\delta $. 
\end{Lemma}

\begin{proof} The proof is straightforward. First we assume that $f \in {\mathcal S}(\R^n)$.
Since $\int_{\R^n} \phi =1$, setting $y=rz,$ we see that,
\begin{eqnarray*}
f\ast \phi_{r}(t)-f(t) & = &  \int_{\mathbb R^{n}} [ f(t-y)- f(t) ] \phi_{r}(y) dy\\ 
                       & = &  \int_{\mathbb R^{n}} [ f(t-rz)-f(t) ] \phi(z) dz \\
                       & = & \int_{\mathbb R^{n}} [ T_{rz}f(t)-f(t) ] \phi(z) dz.
\end{eqnarray*}
Put $h_{r}(t)= f\ast \phi_{r}(t)- f(t);$ and take $0\neq g\in \mathcal{S}(\mathbb R^{n}).$ Then
\begin{eqnarray*}
V_{g}h_{r}(x,w)   &=& \int_{\mathbb R^{n}} V_{g}(T_{rz}f- f )(x,w) \,  \phi(z) dz.\end{eqnarray*}
Taking mixed $L^{p,q}$ norm and an application of Minkowski's inequality for integrals, this gives,
\begin{eqnarray*} \label{many}
\|h_{r}\|_{M^{p,q}} \leq \int_{\mathbb R^{n}} \|T_{rz}f-f\|_{M^{p,q}} \, |\phi(z)| dz.
\end{eqnarray*}
Now the proof follows from the dominated convergence theorem. Note that 
$\|T_{rz}f-f\|_{M^{p,q}} \leq 2 \|f\|_{M^{p,q}}$ by translation invariance of $M^{p,q}$ norm. 

Also since $V_{g}T_{rz}f(x,w)=M_{(0,-rz)}\left(T_{(0,rz)}V_{g}f\right)(x,w),$  we have
\begin{eqnarray*}
\|T_{rz}f \!\!\!&- & \!\!\! f\|_{M^{p,q}} = \|V_{g}T_{rz}f - V_{g}f\|_{L^{p,q}} \\
 & =&  \|M_{(0,-rz)}(T_{(rz,0)}V_{g}f)- M_{(0,-rz)}(V_{g}f)+M_{(0,-rz)}(V_{g}f)-V_{g}f\|_{L^{p,q}}\\
& \leq &  \|T_{(rz,0)}(V_{g}f)-V_{g}f\|_{L^{p,q}} +\|M_{(0,-rz)}(V_{g}f)-V_{g}f\|_{L^{p,q}}
\end{eqnarray*}
each of these tend to $0$ as $r\to 0$, again by the continuity of the  translation and modulation operators in the mixed $L^p$ space $L^{p,q}(\mathbb R^{2n}), ~(1\leq p,q < \infty)$.  

To complete the proof, we note that, if $f$ is a general element in $M^{p,q}(\R^n)$, then by density, we can choose a 
$g\in {\mathcal S}(\R^n)$ such that $\| f-g\|_{M^{p,q}}< \frac{\epsilon}{4}.$ Then 
\Bea \| f \ast \phi_r  \!\!\! &-& \!\!\! f\|_{M^{p,q}} \\
&\leq & \|(f-g) \ast \phi_r \|_{M^{p,q}}+ \|g \ast \phi_r -g\|_{M^{p,q}} + 
\| g-f\|_{M^{p,q}}\\
&\leq &2 \|(f-g) \|_{M^{p,q}}+ \|g \ast \phi_r -g\|_{M^{p,q}}\Eea
in view of Proposition \ref{convoreslt}.  Thus the general case follows since $g \in {\mathcal S}(\R^n)$.
\end{proof}

\begin{Remark}\label{rmkmany}
For future use we record that,  if there are finitely many functions $f_1,...,f_N$, a single $\delta$ can be chosen that works for all $f_i$'s, by simply choosing $\delta = \min\{\delta_i : i=1,2,...N \}.$
\end{Remark}

Some of the modulations spaces $M^{p,q}_{s}(\mathbb R^{n})$ are multiplicative algebras.
To be more specific, we state the following result. For the proof, see  \cite[Proposition 3.2]{nonom},  \cite{bao}, \cite[Corollary 2.7]{ambenyi}.\begin{Proposition}\label{apm}
Let $ X=M^{p,q}_{s}(\mathbb R^{n}), 1\leq p,q \leq  \infty$ and $s> n/q',$ or $X=M^{p,1}_{s}(\mathbb R^{n}),1\leq p \leq \infty, s\geq 0.$ Then $X$ is a multiplication algebra, and we have the inequality 
\bea \label{algineq}
\|f\cdot g\|_{X} \lesssim \|f\|_{X} \|g\|_{X},
\eea
for all $f, g\in X$. 
\end{Proposition}

We end this section with the  following proposition which gives a sufficient condition for a function to be in   $M^{1,1}(\mathbb R^{n})$. See  \cite[p.250]{gro} for a proof. 

\begin{Proposition} \label{ge} Let $L^{2}_{s}(\mathbb R^{n}) =\{f\in L^{2}(\mathbb R^{n}):\int_{\mathbb R^{n}}|f(x)|^{2}(1+|x|)^{2s}<\infty \}.$
If both $f$ and  $\widehat{f}$ are in $L^{2}_{s}(\mathbb R^{n})$ for some $s>n,$ then $f\in M^{1,1}(\mathbb R^{n}).$
\end{Proposition}

\section{Functions operating on $M^{p,1}$}
\label{tpl}

In this section we mainly prove two results on functions operating on $M^{p,1}(\R^n)$, Theorem \ref{orao} asserts that any function that operates in $M^{p,1}(\mathbb R^{n})$ has to be real analytic, and the converse is given by Theorem \ref{convs} for $p=1$. We start with the following

\begin{Definition}
\label{red}
A complex valued function $F,$ defined on  an open set  $E$ in the plane $\mathbb R^{2}$, is said to be real analytic on $E$, if to every point $(s_{0}, t_{0}) \in E,$ there corresponds an expansion of the form
$$F(s, t)= \sum_{m,n=0}^{\infty} a_{mn} \, (s-s_{0})^{m} \, (t-t_{0})^{n}, \hskip.1in a_{mn} \in \C$$ 
which converges absolutely for all $(s,t)$ in some neighbourhood of $(s_{0}, t_{0}).$

If $E=\R^2$ and if the above series converges absolutely for all $(s,t) \in \R^2$, then $F$ is called real entire. In that case
$F$ has the power series expansion 
\bea\label{pwrexp} F(s, t)= \sum_{m,n=0}^{\infty} a_{mn}\, s^{m} \, t^{n}\eea
that converges absolutely for every $(s,t) \in \R^2. $
\end{Definition}

\begin{Remark}
If $F$ is real analytic at a point $(s_0,t_0) \in \R\times \R$, then the above power series expansion shows that $F$ has an analytic extension $F(s+is', t+ it')$ to an open set in the complex domain $ \C \times \C$ containing $(s_0,t_0)$. Also, if $F$ is real analytic in an open set in $\R^2$, then 
fixing one variable, $F$ is a real analytic function of the other variable.  
\end{Remark}

\begin{Remark}
Note that $F$ is real analytic everywhere on $\R^2$, does not imply that $F$ is real entire. A standard example is the 
function $F(x,y)= \frac{1}{(1+x^2)(1+y^2)}$ which is real analytic everywhere on $\R^2$, but the power series 
expansion  around $(0,0)$, converges only in the unit disc $x^2 + y^2<1$.  
\end{Remark}
\noindent
{\em Notation.} If $F$ is real entire function given by \eqref{pwrexp}, then we denote by  $\tilde{F}$ the function given by the power series expansion
\bea\label{tildenotn} \tilde{F}(s,t) = \sum_{m,n=0}^{\infty} |a_{mn}|\, s^{m} \, t^{n}.\eea
Note that $\tilde{F}$ is real entire if $F$ is real entire. Moreover, as a function on $[0,\infty) \times [0,\infty)$, it is monotonically increasing  with respect to each of the variables $s$ and $t$.

\begin{Theorem}\label{orao}
Suppose that $F$ is a complex valued function on $\R^2.$ 
If $F$ operates in $M^{p,1}(\mathbb R^{n}), 1\leq p \leq \infty $, then $F$ is real analytic on $\R^2$. 
 Moreover, $F(0)=0$ if $1\leq p < \infty $.
\end{Theorem}

Before proving this theorem, we discuss some interesting consequences of this result. 
First notice that for $\alpha > 0$, the complex function $$F(z) = |z|^\alpha z= (x^2 + y^2)^{\frac{\alpha}{2}} (x+iy),$$ 
as a mapping from $\R^2 \to \R^2$
may be written as
$$F(x,y)= \left( (x^2 + y^2)^{\alpha/2} x, (x^2 + y^2)^{\alpha/2} y\right). $$ Note that the functions $(x,y) \mapsto (x^2 + y^2)^{\alpha/2} x$ and
$(x,y) \mapsto (x^2 + y^2)^{\alpha/2} y$ are real analytic at zero only if $\alpha \in 2 {\mathbb N}$.
Thus the above theorem answers negatively, the open question raised in \cite{ruzha} 
regarding the validity of an inequality of the form $$\| |u|^\alpha u \|_{M^{p,1}} \lesssim \|u\|_{M^{p,1}}^{\alpha +1},$$ 
for all $u \in {M^{p,1}}(\R^n)$, for  $\alpha \in (0, \infty)\setminus 2 \mathbb N.$ In fact we have the following
\begin{cor}
\label{op}
 There exists $f\in M^{p,1} (\R^n)$ such that $f|f|^{\alpha} \notin M^{p,1}(\R^n)$, for any $\alpha \in (0, \infty) \setminus  2\mathbb N.$ 
\end{cor}
\begin{proof}
If possible, suppose that $F(f)\in M^{p,1}(\mathbb R^{n})$  for all $f\in M^{p,1}(\mathbb R^{n})$, where $F:\C (\approx \R^{2}) \to \mathbb C$ given by 
 $F(z)= z|z|^{\alpha} = x(x^2 +y^2)^{\alpha/2} +iy (x^2 +y^2)^{\alpha/2},$  for  $\alpha \in (0,\infty) \setminus 2\mathbb N.$ But then by Theorem \ref{orao}, $F$ must be real analytic on $\mathbb R^{2},$ which is absurd.
\end{proof}

\begin{cor}
\label{cor1}
If $f\in M^{p,1} (\R)$ then $|f|$ need not be in $M^{p,1}(\R)$. Conversely 
$|f| \in M^{p,1}(\R)$ does not imply that $f\in M^{p,1} (\R)$.\end{cor}

\begin{proof}
The function  $F(z) =|z| =(x^2 +y^2)^{ 1/2 }$ is not real analytic on $\C \approx \R^2$, which shows the first part.
For the converse, consider the function $f:\mathbb R \to \mathbb R$ given by 
$$ f(x)=\begin{cases}
1-x & \text{if}  \ 0\leq x <1,\\
-1+ x, & \text {if} \  -1\leq x <0,\\
0, & \text{if}\  |x|\geq 1.
\end {cases}$$
Note that $f$ is discontinuous and hence does not belong to $M^{p,1} (\mathbb R);$ as $M^{p,1}(\mathbb R)\subset C(\mathbb R).$ But $|f|=(1-|x|)_+$, which is the triangle function,  with Fourier transform 
$\left(\frac{\sin(\pi w)}{\pi w}\right)^{2}$. Thus by  Proposition \ref{ge}, $|f| \in M^{1,1}(\mathbb R) \subset M^{p,1}(\mathbb R)$. 
\end{proof}

Now we proceed to prove Theorem \ref{orao}. Our proof is motivated  by a classical result of
Helson, Kahane, Katznelson and Rudin, for abstract Fourier algebras \cite[p.156]{rudin}, see also, \cite[Theorem 6.9.2]{rb}. Here we restate it for the 
special case of the Fourier algebra $A(\mathbb T^{n})$,  the space  of  
functions on the $n-$torus $\T^n$ having absolutely convergent Fourier series:
$$A(\mathbb T^{n})=\{f:\mathbb T^{n}\to \mathbb C:\sum_{m\in \mathbb Z^{n}} |\hat{f}(m)|< \infty \},$$
where $\hat{f}(m)=\int_{\T^n}f(x)e^{-2\pi i m\cdot x} dx$, the $m$th  Fourier coefficient of $f$. $A(\mathbb T^{n})$ is a Banach algebra under pointwise addition and  multiplication, with respect to the norm
$$\|f\|_{A(\mathbb T^{n})}:=\sum_{m\in \mathbb Z^{n}}|\hat{f}(m)|.$$

\begin{Theorem}[Helson-Kahane-Katznelson-Rudin] \label{HKKR}
Let $F$ be a complex function defined on an  open set $E \subset \R^2$ containing the origin.  If $F$ operates in the Fourier algebra $A(\T^n)$, i.e. $T_F(A(\T^n)) \subset A(\T^n)$, then $F$ is analytic on $E$.
\end{Theorem}
We use the above theorem with $E= \R^2$. We also need the following
result of B\'enyi-Oh, see  \cite[Proposition B.1]{Tbenyi}.
\begin{Proposition}\label{B1}
 Let $f\in M^{p,1}(\mathbb R^n), 1\leq p \leq \infty$ and $\phi$  a smooth function supported on
 $[0,1)^n$. Then $\phi f \in A( \T ^n) $ and satisfies the inequality $$\| \phi f \|_{ A(\T^n) } \lesssim  \| f \|_{M^{p,1} }.$$
\end{Proposition}
Note that the above estimate is stated in \cite{Tbenyi} in terms of the Fourier-Lebsgue spaces,     ${\mathcal F}L^{s,q}(\T^n)$, which coincides with $A(\T^n)$; for $s=0$ and $q=1$, see also \cite[Proposition 2.1, Remark 4.2]{moT}. The next proposition is the main ingredient in the proof of Theorem \ref{orao}.
We first prove the following result.

\begin{Lemma}\label{ftmre}
Let $f$ be a periodic function on $\R^n$ with absolutely convergent Fourier series. Then $f$ is a tempered distribution on $\R^n$ and the Fourier transform of $f$ is the discrete measure $\mu= \sum_{m \in \mathbb Z^{n}} \hat{f}(m) \, \delta_{m}$, where $\hat{f}(m)$ denotes the $m$th Fourier coefficient of $f$, and $\delta_m$ the Dirac mass at $m \in \R^n.$ 
\end{Lemma}

\begin{proof}

Note that  $f$  is continuous on the torus $\T^n$ since the Fourier series is absolutely convergent. Thus $f$ viewed as a periodic function on $\R^n$, is bounded and hence defines a tempered distribution.   

We have $f(x) = \sum_{m \in \Z^n} \hat{f}(m) \, e^{2 \pi i m \cdot x} $ for all $x \in \R^n$.
Thus for $\varphi \in {\mathcal S}(\R^n)$, 
$$ \int_{\R^n} f \, \widehat{\varphi} = \sum_{m \in \Z^n} \hat{f}(m) \int_{\R^n} e^{2 \pi i m \cdot x} \, \widehat{\varphi} (x) \, dx =
\sum_{m \in \Z^n} \hat{f}(m) \, \varphi(m).$$
Writing $\varphi(m) = \delta_m(\varphi)$, this shows that
$\langle \widehat{f}, \varphi \rangle = \left\langle \sum_{m \in \Z^n} \hat{f}(m) \, \delta_m, \varphi \right\rangle$ for all $\varphi \in  {\mathcal S}(\R^n)$.
Thus the Fourier transform of $f$ as a tempered distribution,  is given by $\widehat{f}= \sum_{m \in \Z^n} \hat{f}(m) \, \delta_{m}$ as asserted. \end{proof}

Note that the $\mu$ defined above is a complex Borel measure on $\mathbb R^{n},$ with total variation norm $\| \mu\|=|\mu|(\mathbb R^{n})= \sum_{m\in \mathbb Z^{n}} |\hat{f}(m)|<\infty.$ 

\begin{Proposition} \label{Prop3.3}
If $F$ operates in $M^{p,1}(\R^n), (1\leq p \leq \infty)$, then $F$ operates in $A(\mathbb T^n).$
\end{Proposition}

\begin{proof}
Let $ f\in A(\T ^n).$ Then $f^{\ast}(x)=f(e^{2 \pi ix_1},\cdots, e^{2 \pi ix_n})$ is a periodic function 
on $\R^n$ with absolutely convergent Fourier series $$f^*(x)=  \sum_{m \in \Z^n} \hat{f} (m)\, e^{2 \pi i m \cdot x} .$$ Choose $g\in C_{c}^{\infty}(\mathbb R^n)$, the space of smooth functions on $\R^n$, with compact support, such that $g\equiv1$ on $Q_n=[0, 1)^n.$ Then  we claim that
$gf^{\ast}\in M^{1,1}(\mathbb R^n) \subset M^{p,1}(\mathbb R^n).$ Once the claim is assumed, by hypothesis,
$F(g  f^{\ast})\in M^{p,1}(\mathbb R^n).$ Note that if $z\in \T^n$, then $z= (e^{2\pi ix_1},\cdots , e^{2\pi ix_n})$ for some $x=(x_1,\cdots ,x_n) \in Q_n$, hence \bea\label{cnct} F(f(z))=F( f^*(x))= F(g  f^{\ast}(x)), ~ \mbox{for}~ x \in Q_n.\eea

Now if $\phi\in C_c^\infty(\T^n)$, then $ g \phi^*$ is a compactly supported smooth function on $\R^n$. Also $\phi(z) = g(x) \phi^*(x)$ for every $x \in Q_n$, as per the notation above and hence 
\bea \label{3.2} \phi(z) F(f)(z)= g(x)\phi^*(x) F(gf^*)(x),\eea for some $ x \in Q_n$.
Thus in view of Proposition \ref{B1}, equation (\ref{3.2}) and the hypothesis, we have 
$$\|\phi F(f)\|_{ A(\T^n)}=\| g \phi F(gf^{\ast})\|_{ A(\T^n)}\lesssim \|F(g  f^{\ast})\|_{M^{p,1}},$$
for every smooth cutoff function $\phi$ supported on $Q_n$. Now by compactness of $\T^n$, a partition of unity argument shows that $F(f) \in A(\T^n)$.

To complete the proof, we need to prove the claim. Since $M^{1,1}(\R^n)$ is invariant under Fourier transform, enough to show that 
$\widehat{gf^{\ast}}= \widehat{g} \ast \widehat{f^{\ast}} \in M^{1,1}(\mathbb R^n).$ By Lemma \ref {ftmre},
applied to $f^\ast$,  we see that $$\widehat{f^\ast} = \mu = \sum_{m \in \Z^n} \hat{f} (m) \, \delta_{m}.$$ Hence, 
$$  \widehat{g} \ast \widehat{f^{\ast}}=  \sum_{m \in \Z^n} \hat{f} (m)  \,  \widehat{g} \ast \delta_{m}
= \sum_{m \in \Z^n} \hat{f} (m) \,  T_{m} \widehat{g} .$$

Since the translation operator $T_m$ is an isometry on $M^{1,1}(\mathbb R^n),$ it follows that the above series is absolutely convergent in $M^{1,1}(\mathbb R^n)$, and hence $\widehat{gf^{\ast}} \in M^{1,1}(\mathbb R^n)$ as claimed.
\end{proof}

\begin{proof}[Proof of Theorem \ref{orao}]

If $F$ operates in $M^{p,1}(\R^n)$, then $F$ operates in $A(\T^n)$ by Proposition \ref{Prop3.3}. Hence the analyticity follows from  Theorem \ref{HKKR} with $E=\R^2$.

Note that  the zero function $u_0 \equiv 0 \in M^{p,1}(\R^n)$ and $F(u_0)(x)= F(0)$ for all $x \in \R^n$. But the only constant function in $M^{p,1}(\R^n) \subset L^p(\R^n), ~1\leq p<\infty$
is the zero function. It follows that $F(0) = 0$ if $p<\infty$. \end{proof}

Now we prove the following converse to the above theorem and the proof is more interesting.

\begin{Theorem} \label{convs}
Let $F$ be a real analytic function on $\mathbb R^{2}$ with $F(0)=0$. Then  $ F(f)\in M^{1,1}(\mathbb R^{n})$ for all $f \in M^{1,1}(\mathbb R^{n}).$ 
 
Moreover, if $F$ is a real entire function given by  $F(x,y) =\sum_{m,n} a_{mn}  x^n y^n$,  
with $F(0)=0,$ then we also have the estimate
 \bea \label{favest} \|F(f) \|_{X} \lesssim \tilde{F}\left( \|f_1 \|_{X},  \|f_2\|_{X}\right) ,~ f=f_1+if_2\eea
for all $f\in X,$ where 
$ X$ denotes $M^{p,1}_{s}(\mathbb R^{n}),~1\leq p \leq \infty, s\geq 0,$ or $X=M^{p,q}_{s}(\mathbb R^{n}), ~1\leq p,q \leq  \infty, s>n/q',$
and $\tilde{F}(x,y)$ is the real entire function given by   $\tilde{F}(x,y)= \sum_{m,n} |a_{mn} | x^n y^n.$
 \end{Theorem}
 
\begin{Remark}
Corollary 3.3 of \cite[p. 355]{nonom} is a particular case of Theorem \ref{convs}; as every complex-entire function is real entire as a function on $\mathbb R^{2}.$
\end{Remark}

For arbitrary real analytic function $F$, we do not have a favourable estimate like (\ref{favest}). In fact, we prove $F(f) \in M^{1,1}(\mathbb R^{n})$
by a different argument. 
We start with the following definition.
\begin{Definition}
Let $f$ be a function defined on $\mathbb R^{n}$, we say that $f$ belongs to $M^{p,1}(\mathbb R^{n})$ locally at a point $x_0\in \mathbb R^{n}$ if there is a neighbourhood $V$ of $x_0$ and a function $g\in M^{p,1}(\mathbb R)$ such that $f(x)= g(x)$ for every $x \in V.$  
We say that $f$ belongs to $M^{p,1}(\mathbb R^{n})$ at $\infty$, if there is a compact set $K\subset \mathbb R^{n}$ and  a function $h\in M^{p,1}(\mathbb R^{n})$ such that $f(x)=h(x)$ for all $x\in \R^n \setminus K.$
\end{Definition}

We denote by $M^{p,1}_{loc}(\mathbb R^{n})$, the space of functions that are locally in $M^{p,1}(\mathbb R^{n})$ at each point $x_0 \in \R^n$. 

\begin{Lemma} \label{phif}

Let $1\leq p \leq \infty.$ A function $f \in M^{p,1}_{loc}(\mathbb R^{n}),$ if and only if $\varphi f \in M^{p,1}(\mathbb R^{n})$ for every $\varphi \in C_c^\infty(\R^n)$.

A function $f$ belongs to $M^{p,1}(\mathbb R^{n})$ at $\infty$, if and only if there exists a $\varphi \in C_c^\infty(\R^n)$ such that
 $(1-\varphi)f \in M^{p,1}(\mathbb R^{n}).$ 
 
 \end{Lemma}

\begin{proof}

If $\varphi f \in M^{p,1}(\mathbb R^{n})$ for all $\varphi \in C_c^\infty(\R^n)$, then $f $ is clearly in $ M_{loc}^{p,1}(\mathbb R^{n})$. In fact
for any point $x\in \R^n$, we can choose a smooth function $\varphi$ with compact support, which has value one in a neighbourhood of $x$, by smooth version of Urysohn lemma, see \cite[p.245]{fol}. Then $f\equiv \varphi f$ in that neighbourhood.

Conversely, suppose $f  \in M^{p,1}_{loc}(\mathbb R^{n})$ and $\varphi \in C_c^\infty(\R^n)$ with support $K$.  By hypothesis,  for each point 
$x \in K$, there is an an open ball $B_r(x)$ of radius $r$ and centered at $x$ such that $ f $ coincides with a  $g  \in M^{p,1}(\mathbb R^{n})$ in that ball. 
By compactness of $K$, we can find finitely many points $x_1,x_2,..x_N$ such that the balls $B_{r_i}(x_i), i=1,2,..,N$ cover $K$. Let $\{\varphi_i: i=1,2,...,N\}$ be a partition of unity subordinate to this cover. 

Let $g_i \in M^{p,1}(\mathbb R^{n})$ be such that $f=g_i$ on $B_{r_i}(x_i)$. Since $\phi_i$ is supported in $B_{r_i}(x_i)$, we also have $\varphi_i f = \varphi_i g_i $ on $B_{r_i}(x_i)$, and $\varphi_i g_i  \in M^{p,1}(\mathbb R^{n})$ since $\varphi_i \in C_c^\infty(\mathbb R^{n})$.  Note that we also have  $ \varphi \, \varphi_i g_i  \in M^{p,1}(\mathbb R^{n})$, since $\varphi \varphi_i$ is also in $C_c^\infty(\mathbb R^{n})$.
Thus $\varphi \, \varphi_i f \in M^{p,1}(\mathbb R^{n})$ for each $i$. But $\sum_{i=1}^N  \varphi_i  =1$, implies
$\varphi f = \sum_{i=1}^N \varphi \, \varphi_i f \in M^{p,1}(\mathbb R^{n})$, thus proves the first part of the Lemma.

Again, if $\varphi \in C_c^\infty(\R^n)$ is such that
 $(1-\varphi)f \in M^{p,1}(\mathbb R^{n}),$ clearly  $f$ coincides with a function in $ M^{p,1}(\mathbb R^{n})$ in the compliment of a compact set, namely the function $(1 - \varphi)f$. On the other hand, suppose there exists a $g \in M^{p.1}(\R^n)$ such that $f=g$ on the complement of a large ball $B(0,R)$ of radius $R$, centered at origin. Let $\varphi$ be a smooth function with support $B(0,R)$. Then $(1-\varphi) \equiv 1 $ on $|x|>R$ and hence $(1-\varphi)f=(1-\varphi)g  = g -\varphi g \in M^{p,1}(\R^n)$, as both $g$ and $\varphi g$ are in $M^{p,1}(\R^n)$. This completes the  proof.
\end{proof}

The following lemma gives a useful test for a function to be in $ M^{p,1}(\mathbb R^{n}).$

\begin{Lemma} \label{l2g}
If $f\in M_{loc}^{p,1}(\R^n)$ and $f$ belongs to $M^{p,1}(\mathbb R^{n})$ at infinity, for $1\leq p \leq \infty$, then  $f \in M^{p,1}(\mathbb R^{n})$.
\end{Lemma}

\begin{proof}
Since $f$ belongs to $M^{p,1}(\mathbb R^{n} )$ at infinity, there exists  a $\varphi \in C_c^\infty(\R^n)$ such that $(1-\varphi )f \in 
M^{p,1}(\mathbb R^{n} )$. 
Now  $f= \varphi f + (1-\varphi )f $, and  both $\varphi f$ and $(1-\varphi )f $ are in $M^{p,1}(\mathbb R^{n} )$, 
by  Lemma \ref{phif}. Hence, $f \in M^{p,1}(\mathbb R^{n}).$ This completes the proof.
\end{proof}

Now we proceed to prove Theorem \ref{convs}. We start with the  following technical result.

\begin{Proposition} \label{small}
Let $f\in M^{1,1}(\mathbb R^{n} )$, $x_0 \in \R^n$ and $\epsilon >0$. Then there exists a $\phi \in C_c^\infty(\mathbb R^{n})$ such that  $\| \phi  \left[ f -f(x_0)  \right] \|_{M^{1,1}} < \epsilon$.  The function $\phi $ can be chosen so that $\phi \equiv 1$ in  some neighbourhood of $x_0$.

There also exists a $\psi \in C_c^\infty(\R^n)$ such that $\| (1-\psi) f \|_{M^{1,1}}<\epsilon$. 
\end{Proposition}

\begin{proof} 
 Let $\varphi$ be a smooth function supported in the ball $B_2(0)$ such that $\varphi \equiv 1$ on $B_1(0)$ and set 
 $\varphi^\lambda(x) = \varphi(\lambda x)$. To prove the first part, enough to show that the $M^{1,1}$
 norm of the function
 $ h^\lambda(x):= \varphi^\lambda(x-x_0) [f(x)-f(x_0)]$ tends to zero as $\lambda \to \infty$.
 
For notational convenience, we assume $x_0=0$. Note that 
 \bea \label{multphi} h^\lambda(x) =  \varphi(x)\,  h^\lambda (x), 
 \eea
for $\lambda >2$, as $\varphi \equiv 1$ on the support of $\varphi^\lambda$ in this case. Since the Fourier transform is an isometry on $M^{1,1}(\R^n)$,  enough to estimate $\widehat{h^\lambda}$. Since 
$\widehat{\varphi \,h^\lambda}= \widehat{\varphi} \ast \widehat{h^\lambda},$
in view of (\ref{multphi}) and  Proposition \ref{convoreslt}, we see that 
\Bea
\| \widehat{h^\lambda}\|_{M^{1,1}} &=&  \| \widehat{ \varphi \,h^\lambda } \|_{M^{1,1}}\\
&\leq&\| \widehat{\varphi} \|_{M^{1,1}} \| \widehat{h^\lambda}\|_{L^1}.
\Eea
 
Since $\widehat{h^\lambda}= \widehat{\varphi^\lambda \, f}- f(0) \widehat{\varphi^\lambda} 
  = \widehat{\varphi^\lambda} \ast \widehat{ f} - f(0) \widehat{\varphi^\lambda}$, writing $f(0)= \int_{\R^n} \widehat{ f} (y) dy,$ we see that 
   \Bea \widehat{h^\lambda}(\xi) &=& \int_{\R^n } \widehat{f}(y) \left[ \widehat{\varphi^\lambda}(\xi-y) 
  - \widehat{\varphi^\lambda}(\xi) \right] dy \\
  &=& \int_{\R^n } \widehat{f}(y)\frac{1}{\lambda^n}\left[ \widehat{\varphi}\left(\frac{\xi-y}{\lambda}\right)
   - \widehat{\varphi} \left(\frac{\xi}{\lambda} \right) \right] dy.
  \Eea
  Taking the $L^1$ norm on both sides and by the change of variable $\xi \to \lambda \xi$, we see that 
  \bea
 \int_\xi| \widehat{h^\lambda}| d \xi &\leq&  \int_{\R^n } |\widehat{f}(y)|  
 \int_{\xi } \left| \widehat{\varphi}\left(\xi-\frac{y}{\lambda}\right)
   - \widehat{\varphi}\left(\xi \right) \right| d \xi \, dy \nonumber \\
 \label{revineq}&\leq&  \int_{\R^n } |\widehat{f}(y)|
 \left\| \widehat{\varphi}\left(\cdot- \frac{y}{\lambda}\right)
   - \widehat{\varphi}\left(\cdot \right) \right\|_{L^1}\, dy.  
    \eea
Now we note that $M^{1,1}(\mathbb R^{n}) \subset L^1(\mathbb R^{n})$ and hence $\widehat{f} \in L^1(\R^n)$. Thus the above tends to zero as $\lambda \to \infty$, by dominated convergence theorem and the continuity of the translation in $L^1(\mathbb R^{n})$. 
  
  For general $x_0$, we can continue the same proof by taking $\varphi^\lambda(x-x_0)$ and carrying out the proof as above.
  
  To prove the second part, we choose a $\chi \in C_c^\infty(\R^n)$ with $\chi(0) =1$, and estimate the $M^{1,1} $ norm of $[1- \chi(\lambda x)] f(x)$, for $\lambda >1$. As before, since $M^{1,1}(\mathbb R^{n})$ is invariant under the Fourier transform, enough to estimate the Fourier transform of $[1-\chi(\lambda x) ] f(x)$, which is $\widehat{f}(\xi)- \widehat{f} \ast \varphi_\lambda(\xi)$, with $\varphi =\widehat{\chi}$. This tends to zero in  $M^{1,1}(\mathbb R^{n})$ as $\lambda \to 0$, by Lemma \ref{conv} since $\int \widehat{\varphi} = \varphi(0)=1$.

Now we can choose for $\phi,$ any $\varphi^\lambda$ for  sufficiently small $\lambda$. This completes the proof.
\end{proof}

\begin{Remark}\label{manyfns} If there are finitely many functions $f_1,f_2,...f_N$, then one can choose a single  
$\phi$ and $\psi$ that works for all these functions. All we need to do is to dominate the inequality (\ref{revineq})
with $|\widehat{f}\, |$ replaced by $\sum_1^N|\widehat{f}_i|$, to get a single $\phi$ valid for all $f_i$'s.

On the other hand, if $\psi_i= \varphi_{\lambda_i}$ for $f_i$, then if $\lambda = \min\{\lambda_i, i=1,2,...N\}$, then 
$\psi = \varphi_\lambda$ will work for all $f_i$, as observed in Remark \ref{rmkmany}. 

\end{Remark}

\begin{proof}[Proof of Theorem \ref{convs}]
Write $f= f_1+ if_2 \in M^{1,1}(\R^n)$, where $f_1$ and $f_2$ are real functions, and with an abuse of notation, we write $F(f) = F (f_1, f_2)$.
To show that $F(f)$ is in $M^{1,1}(\R^n)$,  enough to show, in view of Lemma \ref{l2g} that 
$F(f) \in M^{1,1}_{loc}(\R^n)$ and $F(f)$ belongs to $M^{1,1}(\mathbb R^{n})$ at $\infty.$
First we show that $F(f) \in M^{1,1}_{loc}(\R^n)$. 

Fix $x_{0} \in \mathbb R^{n} $ and put $f(x_0)= s_{0} + it_{0}$. 
Since $F$ is real analytic at $(s_0, t_0)$, there exists a $\delta >0 $ such that $F$ has the power  series expansion
\begin{eqnarray} \label{3.4}
\label{mc}
F(s,t)= F(s_{0}, t_{0}) + \sum_{m,n=0}^{\infty} a_{mn} (s-s_{0})^{m} (t-t_{0})^{n},~ ~(a_{00}=0)
\end{eqnarray}
which converges absolutely for $|s-s_{0}|\leq \delta, |t-t_{0}|\leq \delta.$
Then 
\bea
\label{mc2}
F(f_1(x),f_2(x)) & = & F(s_{0}, t_{0})   \nonumber \\
&&+ \sum_{(m,n)\neq (0,0) } a_{mn} [f_1(x)-f_1(x_{0}) ]^{m} [f_2(x)-f_2(x_{0})]^{n}
\eea
whenever the series converges. 

Note that both $f_1$ and $f_2$ are in $M^{1,1}(\mathbb R^{n})$, being the real and imaginary part of $f$. 
Hence by Proposition \ref{small}, and Remark \ref{manyfns}, we can find a $\phi \in C_c^\infty(\R^n)$, such that 
$\phi \equiv 1$ near $x_0$ and $\| \phi [ f_i - f_i(x_0) ]  \|_{M^{1,1}}< \delta $, for $i=1,2$.
Now consider the function $G$ on $\R^n$ defined by 
\bea
G(x) &=& \phi(x) \, F(s_{0}, t_{0}) \nonumber\\
&& + \sum_{(m,n)\neq (0,0) } a_{mn} \left( \phi(x)[ f_1(x)-f_1(x_{0}) ] \right) ^{m} \left( \phi(x) [f_2(x)-f_2(x_{0}) ] \right)^{n}.\eea

Since  $\| \varphi [f_i -f_i(x_0)] \|_{M^{1,1}}< \delta $, for $i=1,2$ and in view of the algebraic inequality (\ref{algineq}), we see that the above series is absolutely convergent in $M^{1,1}(\R^n).$
Also since $\phi \equiv 1$ in some neighbourhood of $x_0$, it follows that $G \equiv F(f)$ in some  neighbourhood of $x_0$.
Since $x_0$ is arbitrary, this shows that $F(f) \in M^{1,1}_{loc}(\R^n)$

To show that $F(f)\in M^{1,1}(\R^n)$ at infinity, we first observe that $f(\infty)=0$, in the following sense:
If $A_R= \{ x_k \in \R^n : |x_k|>R ~\mbox{ and}~ \lim_{x_k\to \infty} f(x_k) \neq 0\}$, then $\lim _{R\to \infty } |A_R| =0$, where $| A |$ denotes the Lebesgue measure of the set $A$. 

Since $f(\infty)=0 $ in the above sense, we take $(s_0,t_0)=(0,0)$ in equation (\ref{3.4}).  Since $F({ 0})=0,$  the expansion \eqref{mc2} now becomes
\Bea
\label{mc}
F(f_1(x),f_2(x)) & = & \sum_{(m,n)\neq (0,0) } a_{mn} \, [f_1(x)] ^{m} \, [f_2(x)]^{n},
\Eea  
whenever the series converges.

By Proposition \ref{small}, we have $\|(1-\psi) f_i \|_{M^{1,1}} < \delta$, for $i=1,2$ for some $\psi \in C_c^\infty(\R^n)$. 
 Now consider the function $H$ defined by
\Bea
H(x)= \sum_{(m,n)\neq (0,0) } a_{mn} \, [(1-\psi (x))f_1(x) ]^m \, [(1-\psi (x))f_2(x) ]^n.\Eea

The above series is absolutely convergent in $M^{1,1}(\R^n)$, in view of the above norm estimates, hence
$H \in M^{1,1}(\R^n)$.
Also since $\psi$ is compactly supported, $1- \psi \equiv 1 $ in the complement of a large ball centered at the  origin, hence $H =F(f)$ in the compliment of a compact set. This shows that $F(f)$ belongs to $M^{1,1}(\R^n)$ at infinity. 

Note that if $F$ is real entire, then $ F$ has a power series about the origin. Also since $F(0)=0,$
the equation (\ref{mc2}) in this case takes the form
\bea
F(f_1(x),f_2(x)) =
\sum_{(m,n)\neq (0,0) } a_{mn} \, [f_1(x) ]^{m}\,  [f_2(x)]^{n}
,\eea
which converges in $X$ for all $f\in X$. Hence, in view of the multiplicative inequality  \eqref {algineq} and \eqref {reimineq}, we see that
\Bea \| F(f)\|_{X} &\lesssim& \sum_{(m,n)\neq (0,0) } |a_{mn}|   \|f_1\|_{X}^m \,  \|f_2\|_{X}^n\\
&&~~~ = \tilde{F} \left(\|f_1\|_{X} , \|f_2\|_{X} \right),
\Eea which is the asserted estimate in the real entire case.
\end{proof}

Now we give the proof of the characterisation of functions operating on $M^{1,1}(\R^n)$.
\begin{proof}[Proof of Theorem \ref{re}]
By Theorem \ref{orao}, if $F$ operates in $M^{1,1}(\mathbb R^{n})$, then $F(0)=0$ and $F$ is real analytic on $\R^2$.
Also by Theorem \ref{convs} if $F$ is real analytic on $\R^2$, and $F(0)=0$, then $F$ operates in $M^{1,1}(\R^n).$
It follows that $F$ operates in $M^{1,1}(\mathbb R^{n})$ if and only if $F$ is real analytic on $\R^2$ and $F(0)=0.$
\end{proof} 

\section{Applications to dispersive equations}
As an application of Theorem \ref{convs}, we study in this section, the local well-posedness of the initial value problems for some dispersive equations in the modulation spaces  $M^{p,1}_{s}(\R^{n})$ for $ 1\leq p \leq \infty, s\geq 0$ and $ M^{p,q}_{s}(\R^{n})$ for $1\leq p,q \leq  \infty, s>n/q'$. We use the generic notation $X$, for the above modulation spaces. 
Specifically, we study the following initial value problems:
 \begin{equation}
\label{nls}
(NLS) \  i\frac{\partial u}{\partial t}+\bigtriangleup_{x}u= F(u),  ~u(x, t_0)=u_{0}(x),
\end{equation}
\begin{equation}
\label{nlw}
(NLW) \  \frac{\partial^{2} u}{\partial t^{2}}- \bigtriangleup_{x}u = F(u), ~u(x,t_0)=u_{0}(x), \frac{\partial u}{\partial t}(x,t_0)=u_{1}(x), 
\end{equation}
\begin{equation}
\label{nlkg}
(NLKG) \ \frac{\partial^{2}u}{\partial t^{2}} + (I-\bigtriangleup_{x})u =F(u), ~u(x,t_0)=u_{0}(x), \frac{\partial u}{\partial t}(x,t_0)=u_{1}(x),
\end{equation} 
where $t_0 \in \R$, $u_0, u_1$ are complex valued functions on $\R^n$, $I$ is the identity map
and $F$ is a real entire function with $F(0)=0$. 
Our approach is via standard fixed point argument using Banach's contraction principle.

We start with the observation that the partial derivatives $\partial_xF(x,y)$ and $\partial_y F(x,y)$ are  real entire functions if $F$ is real entire. This can be easily seen from the power series expansion $F(x,y)= \sum_{m,n=0}^\infty a_{mn} x^m \, y^n, ~ (x,y) \in \R^2$. In fact we can do term by term differentiation and get \bea \label{dxf}\partial_xF(x,y) = \sum_{m\geq 1,n\geq 0} m \, a_{mn} \, x^{m-1} \, y^n.\eea
This is justified because the above power series is absolutely convergent on $\R^2$: In fact, we have  $m \leq 2^{m-1}$ for 
$m \geq 1$ and hence  $$|m\, x^{m-1}| \leq  (2 |x|)^{m-1} \leq (1+2|x|)^m.$$
Thus $\sum_{m\geq 1, n \geq 0} m\, |a_{mn}| \, |x|^{m-1} \, |y|^n \leq \tilde{F}(1+2|x|,|y|) <\infty$ for all $(x,y) \in \R^2$. Similarly, we can also show that $\partial_yF$ is real entire and has the  expansion
\bea
 \label{dyf}\partial_yF(x,y) = \sum_{m\geq 0,n\geq 1} n \, a_{mn} \, x^m \, y^{n-1}\eea
valid for all $(x,y)\in \R^2$. From \eqref{dxf} and \eqref{dyf} 
we also get the inequalities \Bea |\partial_xF(x,y)| &\leq&  \widetilde{\partial_xF}(|x|,|y|) :=\sum_{m\geq 1,n\geq 0} m |a_{mn}|| x|^{m-1} \, |y|^n,\\
|\partial_yF(x,y)| &\leq&  \widetilde{\partial_yF}(|x|,|y|) := \sum_{m\geq 0,n\geq 1} n |a_{mn}| |x|^m \, |y|^{n-1}.\Eea

Note that we cannot expect a similar inequality by replacing $x$ and $y$ by functions $u$ and $v$ in the Banach algebra $X$ as $\partial_xF(u)$ and $\partial_y F(u)$ need not be in $X$ because of the possible nonzero constant term in the power series expansion. However, we have the following substitute given in the following 

\begin{Lemma} \label{Lmdf}
Let $F$ be a real entire function on $\R^2$, then the partial derivatives $\partial_xF(x,y)$ and $\partial_y F(x,y)$ are also real entire functions. Moreover, if $u =u_1+iu_2 \in X$, the modulation space mentioned above, then the following estimates hold 
\bea 
\|w \partial_x F(u_1,u_2)\|_X &\lesssim& \|w \|_X \, 
\widetilde{\partial_xF}(\|u_1\|_X,\|u_2\|_X),\\
\|w \partial_y F(u_1,u_2)\|_X &\lesssim& \|w\|_X
\widetilde{\partial_yF}(\|u_1\|_X,\|u_2\|_X)\eea
for every $w \in X$.
\end{Lemma}

\begin{proof} We have already observed that $\partial_xF(x,y)$ and $\partial_y F(x,y)$ are real entire functions with absolutely convergent  power series expansions \eqref{dxf} and \eqref{dyf} valid  for all $(x,y) \in \R^2$. Now we observe that the series  $$\sum_{m\geq 1,n\geq 0} m \, a_{mn} \, w\,u_1^{m-1} \, u_2^n$$ is absolutely convergent in $X$ for every $w \in X$. In fact,  since $X$ is an algebra, we have  $$\| w \, u_1^{m-1} \,  u_2^n \|_X \lesssim \|w\|_X  \|u_1\|_X^{m-1} \, \| u_2\|_X^n$$ by \eqref{algineq}. It follows that $w\, \partial_x F(u_1,u_2) \in X$ and 

\begin{eqnarray}
\label{dxf3}
\|w\, \partial_x F(u_1,u_2)\|_X &\lesssim& \sum_{m\geq 1,n\geq 0} m |a_{mn}| \, \|w\|_X \, \|u_1\|_X^{m-1} \, \| u_2\|_X^n\\
& = & \|w\|_X \,\widetilde{\partial_xF}(\|u_1\|_X,\|u_2\|_X).\nonumber
\end{eqnarray}
Similarly, $w\, \partial_y F(u_1,u_2) \in X$ and 
\begin{eqnarray}
\label{dyf3}
\| w \, \partial_y F(u_1,u_2)\|_X & \lesssim & \sum_{m\geq 0,n\geq 1} n |a_{mn}| \, \|w\|_X \, \|u_1\|_X^{m} \, \| u_2\|_X^{n-1}\\
& = & \|w\|_X \, \widetilde{\partial_yF}(\|u_1\|_X,\|u_2\|_X).\nonumber
\end{eqnarray}
Hence, the lemma.  \end{proof}

The following proposition gives the essential estimate required to establish the contraction estimate. 
\begin{Proposition} \label{Fu-Fv}
Let $F$ be a real entire function on $\R^2$ and $X$ be the modulation space as in Lemma \ref{Lmdf}. Then we have \Bea \|F(u_1,u_2) \!\!\!\! &-&\!\!\!\! F(v_1,v_2)\|_X \\
&& \lesssim 2 \|u-v\|_X   \left[ 
  \left(\widetilde{\partial_x F} +\widetilde{\partial_y F} \right) (\|u\|_X+\|v\|_X, \|u\|_X+\|v\|_X) \right]\Eea 
for every $u,v\in X$.\end{Proposition}

\begin{proof}
Let $u,v \in X$, with $u=u_1 + iu_2$ and $v=v_1+iv_2$, where $u_1 = \text{Re}(u)$ and   $v_1 = \text{Re}(v)$. Using the formula  
$$ F(x, y)- F(x',y') = \int_0^1 \frac{d}{ds} \left[  F(x'+s(x-x'), y'+s(y-y'))\right] ds$$
for $x,x',y,y' \in \R$, we see that
\bea
&& \!\!\!\!\! \!\!\!\!\! \!\!\!\!\! \!\!\!\!\! F(u_1,u_2) - F(v_1,v_2) \\
&=& \int_{s=0}^1(u_1-v_1) \, \partial_xF (u_1+s(u_1-v_1), u_2+s(u_2-v_2)) \, ds \nonumber\\
&+&   \int_{s=0}^1(u_2-v_2) \, \partial_y F (u_1+s(u_1-v_1), u_2+s(u_2-v_2)) \, ds. \nonumber \eea
Taking norm on both sides and applying Minkowski's inequality for integral, and the Lemma \ref{Lmdf}, 
with $w= v_i-u_i, ~i=1,2$ we get,
\bea \label{fu-v}
&& \|F(u_1,u_2) - F(v_1,v_2)\|_X \\
&\lesssim &\|u_1-v_1\|_X \int_{s=0}^1 \widetilde{\partial_x F} (\|(u_1+s(u_1-v_1))\|_X, \|u_2+s(u_2-v_2)\|_X) ds \nonumber\\
&+&  \|u_2-v_2\|_X \int_{s=0}^1  \widetilde{\partial_y F} (\|u_1+s(u_1-v_1)\|_X, \|(u_2+s(u_2-v_2)\|_X) ds. \nonumber \eea
Note that 
\Bea \| u_i +s(v_i-u_i)\|_X \leq (1-s) \| u_i\|_X + s\| v_i\|_X \leq  \| u\|_X + \| v\|_X\Eea for $i=1,2$ in view of \eqref{reimineq}. Thus using the monotonicty of  
$\widetilde{\partial_xF}$ and $\widetilde{\partial_yF}$ on $[0, \infty)$ in each of its variables, the above integrands are dominated by 
\Bea    \left(\widetilde{\partial_x F} +\widetilde{\partial_y F} \right) (\|u\|_X+\|v\|_X, \|u\|_X+\|v\|_X).
\Eea
In view of these observations,  \eqref{fu-v} yields the estimate
\Bea \|F(u_1,u_2) &-& F(v_1,v_2)\|_X \lesssim ( \|u_1-v_1\|_X) + \|u_2-v_2\|_X) \\
&&\times   \left(\widetilde{\partial_x F} +\widetilde{\partial_y F} \right) (\|u\|_X+\|v\|_X, \|u\|_X+\|v\|_X).\Eea 
 Since $u_1 - v_1 = \text{Re} (u-v)$ and $u_2-v_2 = \text{Im} (u-v)$, the required inequality follows from this, in view of \eqref{reimineq}.
 \end{proof}

The estimates for the linear propagators associated to the Schr\"{o}dinger, the wave  and the Klein-Gordon equations are given by the multiplier theorems on modulation spaces $M^{p,q}_s(\R^n)$,  for three sets of multipliers listed below. 

For a bounded measurable function $\sigma$ on $\R^n$, let  $H_{\sigma}$ denote the Fourier multiplier operator given by
\bea \label{fmulti} H_{\sigma}f(x)= \int_{\R^n} \sigma(\xi)\, \widehat{f}(\xi) \, e^{2\pi i \xi \cdot x} \, d\xi,\eea
for $f \in {\mathcal S(\R^n)}$. The function $\sigma$ is called the multiplier. Here we are concerned with the following families of multipliers defined on $\R^n$: 

\begin{enumerate}
\item \label{mult1}$\sigma(\xi) = e^{-i t4\pi^2  |\xi|^2}$,
\item \label{mult2}$\sigma^{1} (\xi) = \sin(2 \pi t|\xi|)/ 2 \pi |\xi|, ~~~$  $\sigma^{2}(\xi)=\cos(2 \pi t|\xi|)$,
\item \label{mult3}$\mu_1(\xi)= \sin [t(1+|2 \pi \xi|^{2})^{1/2}]/ (1+|2 \pi \xi|^{2})^{1/2},~  \mu_2(\xi)= \cos [t(1+|2 \pi \xi|^{2})^{1/2}].$ \end{enumerate}

We use the following results, and  we refer to  \cite[Lemma 2.2]{ambenyi},  and also \cite[Theorem 1, Corollary 18]{benyi}, for the proof of these facts.

\begin{Proposition} 
\label{pe}

Let $\sigma $ be as in \eqref{mult1} and $H_{\sigma}$ be the Fourier multiplier as in (\ref{fmulti}). Then $H_{\sigma}$ extends to a bounded operator on $M^{p,q}_{s}(\R^n)$ for $1\leq p,q \leq \infty, s \geq 0$.
Moreover, $H_{\sigma}$ satisfy the inequality 
\bea \label{multest1}\| H_{\sigma } f\|_{M^{p,q}_{s}}  \leq c_n (1+t^{2})^{n/4} \|  f\|_{M^{p,q}_{s}}
 \eea
 for some constant $c_n$. \end{Proposition}

\begin{Proposition}\label{wpe}
Let $\sigma^{1}$ and $\sigma^{2}$, be as in \eqref{mult2}.
Then the corresponding Fourier multiplier operators $H_{\sigma^{1}}, H_{\sigma^{2}}$ can be extended as a bounded operators on $M^{p,q}_{s}(\mathbb R^{n})$ for $1 \leq p,  q \leq \infty, s \geq 0$. Moreover, they satisfy the inequalities
\bea \label{multest2}\| H_{\sigma^i } f\|_{M^{p,q}_{s}}  \leq c_n(1+t^{2})^{n/4} \|  f\|_{M^{p,q}_{s}}. \eea
\end{Proposition}

\begin{Proposition}
\label{kgpe}
Let $\mu_1 $ and $\mu_2$ be as in \eqref{mult3}. Then the Fourier multiplier operators $H_{\mu_i}, i=1,2$ can be extended as  a bounded operators on $M^{p,q}_{s}(\mathbb R^{n}),$ for $1 \leq p,  q \leq \infty, s \geq 0$.  Moreover, these operators  satisfy the inequalities
\bea  \label{multest3}\| H_{\mu_i } f\|_{M^{p,q}_{s}}  \leq c_n (1+t^{2})^{n/4} \|  f\|_{M^{p,q}_{s}}. \eea
\end{Proposition}

Now we proceed to prove the well-posedness results, starting with nonlinear Schr\"{o}dinger equation.

\subsection{The nonlinear Schr\"odinger equation}

\begin{Theorem}
\label{nlsrt}
Assume that $u_{0}\in X$ and the nonlinearity $F$ has the form  \eqref{nlf}. Then, there exists 
$T_*=T_*(\|u_{0}\|_{X})<t_0$ and $T^{*}=T^{\ast}(\|u_{0}\|_{X})>t_0$ such that \eqref{nls} has a unique solution $u\in C([T_*, T^{*}];X).$ Moreover, if $|T^{\ast}| < \infty$ then $\limsup_{t\to T^{\ast}} \|u(\cdot, t)\|_{X} = \infty.$ Similarly $\limsup_{t\to T_{\ast}} \|u(\cdot, t)\|_{X} = \infty$,  if $|T_{\ast}|< \infty$.\end{Theorem}

\begin{proof}
We start by noting that \eqref{nls} can be written in the equivalent form
\begin{equation}\label{df}
u(\cdot, t)= S(t-t_0)u_{0}-i\mathcal{A}F(u)
\end{equation}
where
\begin{equation}
S(t)=e^{it\bigtriangleup}, \  (\mathcal{A}v)(t,x)=\int_{t_0}^{t}S(t-t_0- \tau)\, v(t,x) \, d\tau.
\end{equation}
This equivalence is valid in the space of tempered distributions on $\R^n$. 
For simplicity, we assume that $t_0=0$ and prove the local existence on $[0,T]$. Similar arguments also applies to interval of the form $[-T',0]$ for proving local solutions. 

We show that the equivalent integral equation (\ref{df}) has a unique solution, by showing that  
the mapping \begin{equation} \label{inteq}
\mathcal{J}(u)= S(t)u_{0}-i\int_{0}^{t}S(t-\tau) \, [F(u(\cdot, \tau)) ] \, d\tau.
\end{equation}
has a unique fixed point in an appropriate functions space, for small $t$. 
For this, we consider the Banach space $X_{T}=C([0, T]; X)$, with norm
$$\left\|u\right\|_{X_{T}}=\sup_{t\in [0, T]}\left\|u(\cdot, t)\right\|_{X}, \ (u\in X_{T}).$$ 
By the Fourier multiplier estimate \eqref{multest1}
in  Proposition \ref{pe} we see that
\begin{equation*}
\left\|S(t)u_{0}\right\|_{X} \leq c_n (1+ t^{2})^{n/4} \left\|u_{0}\right\|_{X}
\end{equation*}
for $ t \in \R$. It follows that, for $0\leq t \leq	 T$
\begin{equation}
\label{eb}
\left\|S(t)u_{0}\right\|_{X_T} \leq C_{T} \left\|u_{0}\right\|_{X_T} 
\end{equation}
  with $C_{T}=  c_n (1+T^{2})^{n/4}.$

Also, note that if $u \in X_T$, then $u(\cdot, t) \in X$ for each $ t \in[0,  T]$. Hence
by the estimate (\ref{favest}) of Theorem \ref{convs}, $F(u (\cdot,t)) \in X$ and we have 
\bea  \label{xxtest} \| F(u (\cdot,t)) \|_X  &\leq&  \tilde{F}(\|u(\cdot,t)\|_X ,\|u(\cdot,t)\|_X )\\ \nonumber
&\leq &  \tilde{F}(\|u\|_{X_T} ,\|u\|_{X_T} ),\eea where the last inequality follows from 
the fact that $\tilde{F}$, is monotonically increasing on $[0,\infty) \times [0,\infty)$ with respect to each of its 
variables.

Now an application of Minkowski's inequality for integrals, the Fourier multiplier estimate \eqref{multest1} and the estimate \eqref{xxtest}, yields 
\begin{eqnarray}
\label{ed}
\left\| \int_{0}^{t} S(t-\tau) [F(u(\cdot, \tau))]  \, d\tau \right\|_{X} 
& \leq & \int_{0}^{t} \left\|S(t-\tau) [F(u(\cdot, \tau))] \right\|_{X}  \, d\tau \nonumber\\
   &\leq & T C_{T} \, \tilde{F}(\|u\|_{X_T} ,\|u\|_{X_T} )
   \end{eqnarray}
for $0\leq t \leq T$.  Using the estimates \eqref{eb} and \eqref {ed} in \eqref{inteq}, we see that 
\bea
\label{ee}
\left\|\mathcal{J}(u)\right\|_{X_{T}}&\leq& C_{T} \left( \left\|u_{0}\right\|_{X} + T  \tilde{F}(\|u\|_{X_T}, \|u\|_{X_T})\right) \nonumber \\ 
& \leq & C_{T} \left(  \left\|u_{0}\right\|_{X} + T \, \|u\|_{X_T} G(\|u\|_{X_T})  \right)
\eea
where $G$ is a real analytic function on $[0,\infty)$ such that $\tilde{F}(x,x)= x\, G(x)$.
This factorisation follows  from the fact that the constant term in the power series expansion 
for $\tilde{F}$  is zero, (i.e., $\tilde{F}(0,0)=0)$. We also note that $G$ is increasing on $[0,\infty)$.

 For $M>0$, put  $X_{T, M}= \{u\in X_{T}:\|u\|_{X_{T}}\leq M \}$,  which is the  closed ball  of radius $M$, and centered at the origin in  $X_{T}$.  We claim that 
$$\mathcal{J}:X_{T, M}\to X_{T, M},$$
for suitable choice of  $M$ and small $T>0$. 
Note that $C_T \leq C_1$ for $0<T\leq 1$. Hence, putting $M=2 C_{1}\left\|u_{0}\right\|_{X}$, from \eqref{ee} we see that for $u\in X_{T, M}$ and $T\leq 1$
\begin{eqnarray} \label{mapto} 
\left\|\mathcal{J}(u)\right\|_{X_{T}} & \leq & \frac{M}{2}+ T C_1 \,  MG(M) \leq M
\end{eqnarray} 
for  $ T \leq T_1$, where 
\begin{eqnarray} \label{T1} 
T_1= \min \left\{ 1, \frac{1}{2C_1G(M)} \right\}.
\end{eqnarray}
Thus  $\mathcal{J}:X_{T, M}\to X_{T, M},$ for $M=2 C_{1}\left\|u_{0}\right\|_{X}$, and all  $T\leq  T_1$, hence the claim.
  
Now we show that ${\mathcal J}$ satisfies the  contraction estimate 
\bea \label{contra}
\|\mathcal{J}(u)- \mathcal{J}(v)\|_{X_{T}} \leq \frac{1}{2} \|u-v\|_{X_T} \eea
on $X_{T,M}$ if $T$ sufficiently small.

From \eqref{inteq}
and the estimate \eqref{fmulti} in Proposition \ref{pe}, we see that 
\bea\label{ce}
\|\mathcal{J}(u(\cdot, t)- \mathcal{J}(v(\cdot,t))\|_X &\leq& 
\int_{0}^{t}\|S(t-\tau) \, [F(u(\cdot, \tau)) -F(v(\cdot,\tau))]\|_X  \, d\tau. \nonumber \\
&\leq & C_t \int_{0}^{t}\|F(u(\cdot, \tau)) -F(v(\cdot,\tau))\|_X  \, d\tau,
\eea
since $C_{t-\tau}\leq C_t$. By proposition \ref{Fu-Fv} this is atmost $$2 C_t \int_0^t  \|u-v\|_X   \left[ 
  \left(\widetilde{\partial_x F} +\widetilde{\partial_y F} \right) (\|u\|_X+\|v\|_X, \|u\|_X+\|v\|_X) \right] d \tau.$$ Now taking supremum over all $t \in [0,T]$, we see that 
 \bea\label{ju-jv}
 \|\mathcal{J}(u) &-& \mathcal{J}(v)\|_{X_T} \nonumber \\
&\leq& 2TC_T \|u-v\|_{X_T} \left( \widetilde{\partial_x F} +\widetilde{\partial_y F} \right) (\|u\|_{X_T}+\|v\|_{X_T}, \|u\|_{X_T}+\|v\|_{X_T}).\eea
Now if $u$ and $v$ are in $X_{T,M}$, the RHS of \eqref {ju-jv} is at most 
\bea\label{contra28}  2T C_T\|u-v\|_{X_T}  \left( \widetilde{\partial_x F} +\widetilde{\partial_y F} \right) (2M, 2M)\leq  
\frac{ \|u-v\|_{X_T} }{2}\eea
for all $T\leq T_2$, where 
\bea\label{T2}
T_2= \min \left\{1, \left[4 C_1 \left( \widetilde{\partial_x F} +\widetilde{\partial_y F} \right) (2M, 2M)\right]^{-1} \right\}.
\eea

Thus from \eqref{contra28}, we see that the estimate \eqref{contra} holds for all $T<T_2$.  
Now choosing $T^1 = \min \{ T_1,T_2\}$ where $T_1$ is given by  \eqref{T1}, so that
both the inequalities \eqref{mapto} and \eqref{contra} are valid for $T<T^1$. Hence for such
a choice of $T$, ${\mathcal J} $ is a contraction on the Banach space $X_{T,M}$ and hence has a unique fixed point in $X_{T,M}$, by the  Banach's contraction mapping principle. Thus  we conclude that $\mathcal{J}$ has a unique fixed point in $X_{T, M}$ which is a solution of \eqref{df} on $[0,T]$ for any $T<T^1$. Note that 
$T^1$ depends on $\|u_0\|_X$.

The arguments above also give the solution for the initial data corresponding to any given time $t_0$, on an interval $[t_0, t_0 +T^1]$ where $T^1$ is given by the same formula with  $\|u(0)\|_X$ replaced by $\|u(t_0)\|_X$. In other words, the dependence of  the length of the interval of existence on the initial time $t_0$, is only through the norm $\|  u(t_0)\|$. Thus if the solution exists on $[0,T']$ and if $\|u(T')\|_X<\infty$, the above arguments can be carried out again for the initial value problem with the new initial data $u(T')$ to extend the solution  to the larger interval $[0, T'']$. This procedure can be continued and hence we get a solution on maximal interval $[0,T^*]$ having the  following blow up alternative: Either $\| u(\cdot , T^*)\|_X=\infty$ or $\lim_{t\to T^*} \|u(\cdot, t)\|_X = \infty.$ 

Similar arguments can be carried out, to extend the solution to a maximal intervals to the left, of the form 
$[T_*, 0]$. This gives the blow up alternative. 

The uniqueness also follows from the uniqueness of the fixed point for ${\mathcal J}$.
This completes the proof.
\end {proof}

By similar arguments, using the multiplier estimates given in Proposition
\ref{wpe}, Proposition \ref{kgpe}, and using the Proposition \ref {Fu-Fv} to prove contraction estimates,
we can establish analogous local well-posedness results, for the initial value problems for the  wave equation and the Klein-Gordon equation. Instead of repeating the arguments, we only indicate the equivalent 
integral equation in terms of the one parameter groups involved, and the relevant estimates, to carry out the proof as above.

\subsection{The nonlinear wave equation}

\begin{Theorem}
\label{pnlw}
Assume that $u_{0}, u_{1}\in X$ and the nonlinearity $F$ has the form  \eqref{nlf}.  Then, there exists $T^{*}=T^{\ast}(\|u_{0}\|_{X}, \|u_{1}\|_{X})$ such that \eqref{nlw} has a unique solution $u\in C([0, T^{*}]; X).$ Moreover, if $T^{\ast} < \infty$, then $\limsup_{t\to T^{\ast}} \|u(\cdot, t)\|_{X} = \infty.$

\end{Theorem}

\begin{proof}
Equation \eqref{nlw} can be written in the equivalent form

\begin{equation}
u(\cdot, t)= \tilde{K}(t)u_{0} +K(t)u_{1}- \mathcal{B}F(u)
\end{equation}
where
$$
K(t)=\frac{\sin (t\sqrt{-\bigtriangleup})}{\sqrt{-\bigtriangleup}}, \tilde{K} (t) =\cos (t\sqrt{-\bigtriangleup} ), (\mathcal{B}v)(t,x)= \int_{0}^{t} K(t-\tau)v(\tau, x) d\tau.
$$
Consider the mapping
$$\mathcal{J}(u)=\tilde{K}(t)u_{0}+ K(t)u_{1}- \mathcal{B}F(u).$$
By using Proposition \ref{wpe} for the first two inequalities below, and the estimate \ref{favest} for the last inequality, we can write,
\begin{equation}
\begin{cases}
\|\tilde{K}(t)u_{0}\|_{X} \leq C_{T} \|u_{0}\|_{X},\\
\|K(t)u_{1}\|_{X} \leq C_{T} \|u_{1}\|_{X},\\
\|\mathcal{B}F(u)\|_{X} \leq  T C_{T} \tilde{F}( \|u\|_{X},  \|u\|_{X}),
\end {cases}
\end{equation}
where $C_T$ is some constant times $(1+T^2)^{n/4}$, as before. Thus the standard contraction mapping argument can be applied to $\mathcal{J}$ to complete the proof.
\end{proof}

\subsection{The nonlinear Klein-Gordon equation}

\begin{Theorem}
\label{pkg}
Assume that $u_{0}, u_{1}\in X$ and the nonlinearity $F$ has the form  \eqref{nlf}. Then, there exists $T^{*}=T^{\ast}(\|u_{0}\|_{X},\|u_{1}\|_{X})$ such that \eqref{nlkg} has a unique solution $u\in C([0, T^{*}];X).$  Moreover, if $T^{\ast} < \infty$, then $\limsup_{t\to T^{\ast}} \|u(\cdot, t)\|_{X} = \infty.$ 
\end{Theorem}

\begin{proof} The equivalent form of equation \eqref{nlkg} is 
\begin{equation}
u(\cdot, t) = \tilde{K}(t)u_{0}+ K(t)u_{1} + \mathcal{C}F(u), 
\end{equation}
where now
$$K(t)= \frac{\sin t(I-\bigtriangleup)^{1/2}}{(I-\bigtriangleup)^{1/2}}, \tilde{K}(t)= \cos t(I- \bigtriangleup)^{1/2},  \ (\mathcal{C}v)(t,x) = \int_{0}^{t} K(t-\tau)v(\tau, x) d\tau. $$

By using Proposition \ref{kgpe} and the notations above, we can write

\begin{equation}
\begin{cases}
\|\tilde{K}u_{0}\|_{X} \leq C_{T} \|u_{0}\|_{X},\\
\|K(t)u_{1}\|_{X} \leq C_{T} \|u_{1}\|_{X}, \\
\|\mathcal{C}F(u)\|\leq T C_{T}\tilde{F}( \|u\|_{X},  \|u\|_{X}),
\end{cases}
\end{equation}
Now the standard contraction mapping argument applied to $\mathcal{J}$ gives the proof.
\end{proof}

\noindent
{\bf Some open questions.} 

1. It would be interesting to know, if $F$ operates on $M^{p,1}(\R^n)$ for $p>1$ implies that 
$F$ is real analytic. If yes, this will give a characterisation for functions operating on $M^{p,1}(\R^n)$ 
for $p>1$ as well.

2. We have shown the local well-posedness results for real entire nonlinearities on $M^{p,1}(\R^n)$
for $1\leq p \leq \infty$. Since any real analytic function vanishing at origin, maps $M^{1,1} (\R^n)$ 
to itself, by the Theorem \ref{re},  it is natural to ask, if the local well-posedness can be 
proved in $M^{1,1}(\R^n)$, for real analytic nonlinearities.\\

\noindent
{\textbf{Acknowledgement}:} This work is part of the Ph. D. thesis of the first author. He wishes to thank the Harish-Chandra Research institute, the Dept. of Atomic Energy, Govt. of India, for providing excellent research facility.

\end{document}